\documentclass[11pt,leqno,british]{amsart}
\usepackage{fullpage,epsfig,graphics,amsbsy,amssymb,cancel,bbm}
\usepackage{psfrag,hyperref,color,slashed,}
\usepackage{graphicx}
\usepackage{amsfonts,mathdots,dsfont}
\usepackage{enumitem}
\usepackage{amssymb,amsmath,amscd,mathrsfs}
\usepackage{ifpdf}
\usepackage{eucal}
\usepackage{url}
\usepackage{amsthm}

\usepackage{tikz}
\usetikzlibrary{arrows,chains,matrix,positioning,scopes,snakes}

\makeatletter
\tikzset{join/.code=\tikzset{after node path={%
\ifx\tikzchainprevious\pgfutil@empty\else(\tikzchainprevious)%
edge[every join]#1(\tikzchaincurrent)\fi}}}
\makeatother

\tikzset{>=stealth',every on chain/.append style={join},
         every join/.style={->}}
\tikzset{
    >=stealth',
    punkt/.style={
           rectangle,
           rounded corners,
           draw=black, very thick,
           text width=6.5em,
           minimum height=2em,
           text centered},
    pil/.style={
           ->,
           thick,
           shorten <=2pt,
           shorten >=2pt,}
}

\newcommand{\bea}{\begin{eqnarray}}
\newcommand{\eea}{\end{eqnarray}}
\newcommand{\be}{\begin{equation}}
\newcommand{\ee}{\end{equation}}
\newcommand{\nn}{\nonumber}

\newcommand{\bra}{\langle}
\newcommand{\ket}{\rangle}

\newcommand{\im}{\textrm{Im}\,}
\newcommand{\re}{\textrm{Re}\,}

\newcommand{\comp}{\mathbb{C}}

\newcommand{\li}{\mathrm{Li}\,}

\newcommand{\widecheck}[1]{\stackrel{\!\!\!\vee}{#1}}

\def\go{\omega}

\newtheorem{theorem}{Theorem}[section]
\newtheorem{lemma}[theorem]{Lemma}
\newtheorem{proposition}[theorem]{Proposition}

\newtheorem{definition}[theorem]{Definition}

\newtheoremstyle{remarkstyle}
{}
{10 mm}
{}
{}%
{\bfseries}
{}
{  }
{}%

\theoremstyle{remarkstyle} { \newtheorem{remark}[theorem]{Remark} } 

\begin{document}

\begin{flushright}
{\tt UUITP-23/16}\\
\end{flushright}
\vspace{1cm}

\title[Multiple elliptic gamma functions associated to cones]%
{Multiple elliptic gamma functions associated to cones}

\author{Jacob Winding}
\address[Winding]{Department of Physics and Astronomy, Uppsala university,
Box 516, SE-75120 Uppsala, Sweden}
\email{jacob.winding@physics.uu.se}
\begin{abstract} 
We define generalizations of the multiple elliptic gamma functions and the multiple sine functions, associated to good rational cones. 
We explain how good cones are related to collections of $SL_r(\mathbb{Z})$-elements and prove that the generalized multiple sine and multiple elliptic gamma functions enjoy infinite product representations and modular properties determined by the cone. 
This generalizes the modular properties of the elliptic gamma function studied by Felder and Varchenko, and the results about the usual multiple sine and elliptic gamma functions found by Narukawa. 
 \end{abstract}
 
\maketitle
\setcounter{tocdepth}{2}


\tableofcontents

\section{Introduction}

This paper deals with a particular generalization of the \emph{multiple elliptic gamma functions}, which are a family of meromorphic functions that themselves generalize the Euler gamma function. 
A well-known member of this family is the elliptic gamma function, which was first introduced by Ruijsenaars \cite{Ruijsenaars} as a solution to a difference equation involving the Jacobi theta function. 
Similar double products also appears frequently in studying statistical mechanics. 
Felder and Varchenko studied this function in \cite{FelderVarchenko}, motivated by the study of hypergeometric solutions of elliptic qKZB difference equations; and their most striking result is the modular three-term relation that the elliptic gamma function satisfies.
Nishizawa \cite{Nishizawa} then constructed a family of meromorphic functions called the multiple elliptic gamma functions, which include the Jacobi theta function as well as the elliptic gamma function. 
We denote these as $G_r(z| \tau_0, \ldots, \tau_r)$ and they can be represented as the following infinite product 
\[
	G_r ( z | \tau_0, \ldots, \tau_r) = \prod_{n_0,\ldots,n_r = 0}^{\infty} \left (1-e^{2\pi i (- z + (n_0+1)\tau_0 + \ldots +(n_r+1)\tau_r ) } \right ) \left ( 1 - e^{2\pi i (z+n_0\tau_0 + \ldots + n_r \tau_r ) } \right )^{(-1)^r }, 
\]
when $\mathrm{Im}~ \tau_j > 0 \ \forall j$.
The Jacobi theta function corresponds to $G_0$ and the elliptic gamma function to $G_1$. 
It was subsequently shown by Narukawa \cite{Narukawa} that all the functions in this family enjoy modular properties
\[ \begin{split}
	G_r ( z | \underline \tau) =& \exp \left \{ \frac{2\pi i }{(r+2)!} B_{r+2,r+2} ( z | (\underline \tau,-1 ) ) \right \} \nn \\
	&\times  \prod_{k=0}^r G_r \left ( \frac{z}{\tau_k} \right | \left. \frac{\tau_0}{\tau_k},\ldots,\widecheck{ \frac{\tau_k}{\tau_k} }, \ldots, \frac{\tau_r}{\tau_k}, - \frac{1}{\tau_k} \right ) ,
	\end{split}
\]
where $B_{r+2,r+2}$ is a multiple Bernoulli polynomial, and $\check \ $ means we exclude that argument.

In this paper, we consider a generalization of the multiple elliptic gamma functions where we take a strongly convex rational cone $C$ of dimension $r+1$ and consider the function
\[
	G_r^C ( z | \tau_0, \ldots, \tau_r) = \prod_{n \in C^\circ \cap \mathbb{Z}^{r+1}} \left (1-e^{2\pi i (- z + n_0\tau_0 + \ldots +n_r \tau_r ) } \right )  \prod_{n \in C \cap \mathbb{Z}^{r+1}}\left ( 1 - e^{2\pi i (z+n_0\tau_0 + \ldots + n_r \tau_r ) } \right )^{(-1)^r }, 
\]
when $\mathrm{Im} \underline \tau \in \check C^\circ$, where $\check C$ is the dual cone of $C$ and $C^\circ$ means the interior of $C$. 
We call this the generalized multiple elliptic gamma function associated with $C$. 
The motivation for this definition comes from physics, where this special function appears when computing instanton partition functions for gauge theories placed on the toric manifold whose moment map cone is $C$. 
This gives some hints that these functions might be related to invariants of toric manifolds, which is interesting and motivates our studies. 
In general the connections between special function with some modular properties and topological invariants seems like a worthwhile area of investigation.
We will explain a bit more of the connection with physics later in this introduction, but first we will state the main results of the paper.
As is explained in detail in section \ref{sec:goodcones}, a good cone $C$ of dimension $r$ defines a set of $SL_{r}(\mathbb{Z})$ elements, in particular we can choose an element $\tilde K_\rho$ for each generating ray $\rho$ of $C$. Then the ordinary multiple elliptic gamma functions satisfy the following modular property
\[
	\prod_{\rho \in \Delta_1^C} G_{r-2} ( \frac{z}{ (\tilde K_\rho \underline \tau)_1 } | \frac{ (\tilde K_\rho \underline \tau )_2 } { (\tilde K_\rho \underline \tau)_1 }, \ldots, \frac{ (\tilde K_\rho \underline \tau )_{r} } { (\tilde K_\rho \underline \tau)_1 } )
	= \exp \left [ - \frac{2\pi i }{r!} B_{r,r}^C ( z | \underline \tau ) \right ] ,
\]
where $B_{r,r}^C$ is a generalized Bernoulli polynomial associated to $C$ which we will define below, and $\Delta_1^C$ is the set of generators of $C$. 
This identity include the usual modular properties of $G_{r-2}$ if we choose $C = \mathbb{R}^r_{\geq 0}$. 
Further we show that our generalized multiple elliptic gamma function also enjoy another ``modular'' or perhaps more precisely factorization property, 
\[
	G_{r-1}^C ( z | \underline \tau ) = \exp \left [ \frac{2\pi i }{ (r+1)!} B_{r+1,r+1}^{\hat C} ( z | \underline \tau,-1)\right ] \prod_{\rho \in\Delta^C_1} (SK_\rho)^* G_{r-1} (z|\underline \tau)
\]
where now $K_\rho, S \in SL_{r+1}(\mathbb{Z})$ and they are acting on the parameters $(z|\underline \tau)$ as a fractional linear transformation; for details of their definition and group action, see section \ref{sec:goodcones}.

Additionally we introduce and study a closely related family of functions, the multiple sine functions. 
This is another family of meromorphic functions, denoted $S_r$, $r = 1,2,\ldots$, that generalizes the ordinary sine function (which correspond to $S_1$) and are the classical analogs of the multiple elliptic gamma functions.
The second member of the family, the double sine $S_2$ was first introduced by Shintani \cite{Shintani}; it was also independently introduced in mathematical physics by Faddeev \cite{Faddeev:1993rs} under the name quantum dilogarithm. 
The full family was later defined by Kurokawa \cite{multiplesines}, and Narukawa \cite{Narukawa} then gave them integral representations as well as proved some interesting infinite product representations.
We again define a generalization of the multiple sine functions associated to cones, derive an integral representation and finally prove a similar infinite product representation, where we again find one factor for each generator of the cone $C$.

We will now explain something about the motivation from physics. 
When computing the partition function of a supersymmetric gauge theory placed on a sphere, using so called localization \cite{Pestun:2007rz}, which is essentially the famous Atiyah-Bott fixed point theorem applied to the infinite dimensional setting of quantum field theory, one finds that the results are naturally written in terms of the multiple sine function. 
To be precise, the multiple sine function gives us the perturbative result, which is the leading order result when expanding as a sum over instanton contributions; i.e. the contribution from zero instantons. 
In the abelian theory, gauge group $U(1)$, in 5d, we can also find the full answer, which turns out to be naturally written in terms of the multiple elliptic gamma functions instead \cite{Lockhart:2012vp}. 
Next, we can place the same gauge theories on a much larger class of manifolds: namely any Sasaki-Einstein manifold \cite{Kallen:2012cs}. 
In practice we can however only compute the partition function if the manifold is toric; since we need $U(1)$-actions to apply the localization theorem. 
And when we compute the perturbative result for a toric 5d Sasaki-Einstein manifold, we find that it's built up of the generalized multiple sine that we study in this paper \cite{Qiu:2013aga,Qiu:2014oqa}.  
Again for the abelian gauge theory, we can also compute the full instanton partition function, and we find that it consists of our generalized multiple elliptic gamma function \cite{Qiu:2015rwp}.
So this is what lead us to define these functions and study their properties. 
Their origin also imply that they might have close connections with topological invariants of the associated Sasaki-Einstein manifold and its Calabi-Yau cone: it would be interesting to connect them with for example the Donaldson-Thomas invariants.

The rest of this article is organized as follows. 
In the next part, we introduce the necessary preliminaries, defining the usual multiple elliptic gamma functions and multiple sines and state their most important properties.
In the following section, we recall some facts about cones, review their connection with toric geometry, and establish technology for summing over lattice points in a given cone.
Next, we define our generalized multiple sine functions, give an integral representation of them and prove a factorization property they enjoy, which leads to a generalized modularity property for the usual multiple elliptic gamma functions. 
In the final section, we define our generalized multiple elliptic gamma functions, write an integral representation and prove their modular property quoted above.

{\bf Remark:} This paper is a follow-up of an earlier unpublished manuscript \cite{Tizzano:2014roa} by the present author together with Luigi Tizzano, which introduced the same functions, but where we only managed to prove the results in dimensions $r=2,3$ using some cumbersome combinatorial arguments. 
The present note uses more general methods to generalize to any dimension, and also provides integral representations and proves some new properties of these functions.  
\\\\
{\bf Acknowledgments:} The author wish to express his deep gratitude to Jian Qiu and Maxim Zabzine for a lot of useful discussions and suggestions. He also thanks Konstantina Polydorou for discussions.
The author is supported in part by Vetenskapsr\aa det under grants \#2011-5079 and \#2014-5517, in part by the STINT grant and in part by the Knut and Alice Wallenberg Foundation.

\section{Multiple elliptic gamma functions and multiple sines}

In this section we introduce notation and define the well known multiple sine and multiple elliptic gamma functions, which we will later generalize.

\subsection{Multiple q-shifted factorials}
The multiple $q$-shifted factorials was introduced with two parameters in the early works of F.H. Jackson \cite{Jackson1}, and with a general number of arguments we define it as follows. 
Let $x= e^{2\pi i z }$ and $q_j = e^{2\pi i \tau_j}$ for $z\in\comp$ and $\tau_j \in \comp - \mathbb{R}$, $j=0,\ldots,r$. 
Further let
\[
	\begin{split}
		\underline q &= (q_0,\ldots,q_r), \\
		\underline q^{-}(j) &= (q_0, \ldots, q_{j-1},q_{j+1},\ldots,q_r), \\
		\underline q [j] &= (q_0,\ldots,q_j^{-1},\ldots,q_r ) , \\
		\underline q^{-1}&= (q_0^{-1},\ldots,q_r^{-1}). 
	\end{split}
\]
With these conventions in place, we define the q-shifted factorials.
Their definition depends on the sign of the imaginary part of the $\tau_j$ parameters. 
If $\im \tau_j < 0$, $j=0,\ldots,k-1$ and $\im \tau_j > 0$ for $j=k,\ldots,r$, we define the corresponding q-factorial as 
\be
	(x|\underline q)_\infty = \prod_{j_0, \cdots, j_r = 0}^{\infty} \left (1 - x q_0^{-(j_0+1)} \cdots q_{k-1}^{-(j_{k-1}+1)} q_{k}^{j_{k}}\cdots q_r^{j_r} \right )^{(-1)^k}  . 
\ee
We define this function to be invariant under permutations of the ordering of $\underline q$; or in other words we define it so that it's symmetric in the $\underline q$ parameters, which together with the above define it for all choices of $\tau_j \in \comp - \mathbb{R}$. 
The multiple q-factorials is by this definition a meromorphic function of $z$, that satisfy the following functional equations 
\be \label{eq:qfactproperties}
	(x | \underline q )_\infty = \frac{1}{ ( q_j^{-1} x | \underline q[j] )_\infty } , \ \ \ ( q_j x | \underline q)_\infty = \frac{ ( x | \underline q )_\infty } { ( x | \underline q^{-}(j) )_\infty } . 
\ee

\subsection{Multiple elliptic gamma functions}
We also use the notations
\[
	\begin{split}
		\underline \tau &= (\tau_0,\ldots,\tau_r), \\
		\underline \tau^{-}(j) &= (\tau_0, \ldots, \tau_{j-1},\tau_{j+1},\ldots,\tau_r), \\
		\underline \tau [j] &= (\tau_0,\ldots,-\tau_j,\ldots,\tau_r ) , \\
		| \underline \tau | & = \tau_0+ \tau_1+\ldots+\tau_r, 
	\end{split}
\]
and define the multiple elliptic gamma function as
\bea
	G_r ( z | \underline \tau ) &=& ( x^{-1} q_0 \ldots q_r | \underline q )_\infty [ ( x | \underline q )_\infty ]^{(-1)^r} \\
	&=& ( x^{-1} | \underline q^{-1} )^{(-1)^{r+1} } ( x | \underline q )^{(-1)^r } \label{eq:G2def2} .
\eea
The $G_r$ functions form a hierarchy that include the theta function $\theta_0(z,\tau)  = G_0 ( z | \tau)$, and the elliptic gamma function $\Gamma(z,\tau,\sigma) = G_1 ( z | \tau,\sigma)$. 
The multiple elliptic gamma functions satisfy a number of functional relations including 
\bea
	G_{r}(z+1 | \underline \tau ) &=& G_r(z|\underline \tau), \\
	G_r (z+\tau_j | \underline \tau ) &=& G_{r-1} (z|\tau^{-}(j) ) G_r(z|\underline \tau), \\
	G_r(z|\underline \tau ) &=& \frac{1}{G_r ( z - \tau_j | \underline \tau [j] )},
\eea
which follow from the relations \eqref{eq:qfactproperties}. 
They also satisfy an interesting modular property, proven by Narukawa \cite{Narukawa}, 
\begin{theorem}[Modular properties of $G_r(z|\underline \tau)$ ] \label{thm:GrModularity}
If $r \geq 2$ and $\mathrm{Im} \frac{\tau_j}{\tau_k} \neq 0$, then the multiple elliptic gamma function satisfies the identities
\bea \label{eq:GrModularity}
	G_r ( z | \underline \tau) &=& \exp \left \{ \frac{2\pi i }{(r+2)!} B_{r+2,r+2} ( z | (\underline \tau,-1 ) ) \right \} \nn \\
	&&\times  \prod_{k=0}^r G_r \left ( \frac{z}{\tau_k} \right | \left. \frac{\tau_0}{\tau_k},\ldots,\widecheck{ \frac{\tau_k}{\tau_k} }, \ldots, \frac{\tau_r}{\tau_k}, - \frac{1}{\tau_k} \right ) \\
	&=& \exp \left \{ - \frac{2\pi i }{(r+2)!} B_{r+2,r+2} ( z | (\underline \tau,1 ) ) \right \} \nn \\
	&& \times  \prod_{k=0}^r G_r \left (- \frac{z}{\tau_k} \right | \left. - \frac{\tau_0}{\tau_k},\ldots,-\widecheck{\frac{\tau_k}{\tau_k}}, \ldots, - \frac{\tau_r}{\tau_k}, - \frac{1}{\tau_k} \right ).
\eea
\end{theorem}
This result includes the famous modular property of the Jacobi theta function, as well as the modular property studied by Felder and Varchenko \cite{FelderVarchenko}. Here, $B_{r+2,r+2}$ is a multiple Bernoulli polynomial, which we introduce below in section \ref{sec:bernoullipoly}.

\subsection{Multiple sine functions}
Next, we introduce the multiple sine functions, which we denote $S_r$.
These form a hierarchy that include the ordinary sine as the first case $(S_1)$, and they are closely related to the multiple elliptic gamma functions.

To define them, let $\go_1,\ldots,\go_r \in \comp$ be chosen so that they all lie on the same side of some straight line through the origin.  
If this is the case, we can define the multiple zeta function as the series
\be
	\zeta_r ( s,z | \underline \go ) = \sum_{n \in \mathbb{Z}^r_{\geq 0} } \frac{1}{ ( z + n\cdot \underline \go )^s },
\ee
for $z \in \comp$ and $\re s > r$, where the exponential is rendered one-valued. 
This function is holomorphic in this domain and can be analytically continued to $s\in\comp$. 
We use this fact to define the multiple gamma function, following Barnes \cite{Barnes},
\be
	\Gamma_r ( z | \underline \go )  = \exp \left ( \left. \frac{\partial}{\partial s } \zeta_r ( s,z|\underline \go ) \right |_{s=0} \right ).
\ee
Finally the multiple sine is defined in terms of this gamma as
\be
	S_r ( z | \underline \go ) = \Gamma_r ( z | \underline \go )^{-1} \Gamma_r (| \underline \go | - z | \underline \go)^{(-1)^r }.
\ee
This construction is an example of the so called zeta regularization of an infinite product, see for example \cite{zetaregularization} for a discussion.
We can think of this definition as a way to make sense of the following divergent infinite product: 
\[
	\prod_{n\in\mathbb{Z}^r_{\geq 0}} ( z + n\cdot \underline \go ) \prod_{m \in \mathbb{Z}_{>0}^r } ( -z + m\cdot \underline \go )^{(-1)^{r-1}},
\]
which is the Weirstrass representation of the function. 
As stated in \cite{zetaregularization}, another way of making sense of this product is the following expression
\[
	e^{ Q_r ( z ) } \prod_{ \substack{ 
	n \in \mathbb{Z}^r_{\geq 0} \\
	n \neq ( 0, 0,\ldots, 0) } }	
	( 1 - \frac{z}{n\cdot \underline \go} )\exp \left [ \sum_{j = 1}^r \frac{1}{j} \left ( \frac{z}{n\cdot \underline \go } \right ) \right ] 
	\left \{  \prod_{m \in \mathbb{Z}^r_{>0} } ( 1 - \frac{z}{m \cdot \underline \go } ) \exp \left [ \sum_{j = 1}^r \frac{1}{j} \left ( \frac{z}{m\cdot \underline \go } \right ) \right ] \right \}^{(-1)^{r-1}} ,
\]
where $Q_r(z)$ is some polynomial function. 

The multiple sine function has an interesting infinite product representation, proved in general by Narukawa \cite{Narukawa}:
\begin{proposition} \label{prop:SrProductRep}
If $r\geq 2$ and $\mathrm{Im} \frac{ \go_j}{\go_k} \neq 0$, then $S_r$ has the following infinite product representations. 
\bea \label{eq:SrInfiniteProd}
	S_r ( z | \underline \go ) &=& \exp \left [ (-1)^r \frac{\pi i }{r!} B_{r,r} ( z | \underline \go )  \right ] \prod_{k=1}^r ( x_k | \underline q_k )_\infty  \\
	&=&  \exp \left [ (-1)^{r+1} \frac{\pi i }{r!} B_{r,r} ( z | \underline \go )  \right ] \prod_{k=1}^r ( x^{-1}_k | \underline q^{-1}_k )_\infty
\eea 
where $x_k = e^{2\pi i z /\go_k}$ and $\underline q_k = ( e^{2\pi i \go_1 / \go_k },\ldots, \widecheck { e^{2\pi i \go_k/\go_k} } , \ldots , e^{2\pi i \go_r / \go_k } )$.
\end{proposition}
We will find a very similar property for our generalized multiple sine functions.

\subsection{Generalized $q$-polylogarithms}
When we prove the infinite product representation of the generalized multiple sine function, we need the following function and a related lemma, both which are given in Narukawa \cite{Narukawa}.
Let $x = e^{2\pi i z }$ and $q_j = e^{2\pi i \tau_j}$. 
When $\im z > 0 $ and $\im \tau_j \neq 0$ $(j=0,\ldots, r)$, the generalized q-polylogarithm is defined as 
\be
	\li_{r+2} ( x | \underline q ) = \sum_{n=1}^\infty \frac{x^n}{n \prod_{j=0}^r ( 1 - q^n_j ) } . 
\ee
This series converges absolutely and $\li_{r+2}$ is holomorphic in $z$. 
We also need the following lemma:
\begin{lemma} \label{lem:polylogexp}
	If $x$ and $\underline q$ are as defined above, then
	\be
		(x | \underline q )_\infty = \exp [ - \li_{r+2} ( x | \underline q ) ].
	\ee
\end{lemma}
\begin{proof}
See Narukawa \cite{Narukawa}. 
\end{proof}

\subsection{The multiple Bernoulli polynomials} \label{sec:bernoullipoly}
For $z \in \comp$ and $\go_j \in \comp-\{ 0 \}$ we define the multiple Bernoulli polynomials $B_{r,n}$ through the formal generating series
\be
	(-1)^r \frac{t^r e^{zt } } { \prod_{j=1}^r ( 1 - e^{t \go_j } ) }  =\sum_{n=0}^{\infty} B_{r,n}(z | \underline\go) \frac{t^n}{n!}.
\ee
These polynomials appear in the modular properties of elliptic gamma functions and multiple sine functions, and satisfy a number of functional relations that one can prove from their generating function, such as
\be \label{eq:bernoulliproperties}
\begin{split}
	B_{r,n} ( c z | c \underline \go ) &= c^{n-r} B_{r,n} ( z | \underline \go), \\
	B_{r,n} ( | \underline \go | - z | \underline \go ) &= (-1)^n B_{r,n} ( z | \underline \go ) , 
\end{split}
\ee	
as well as a number of additional similar properties. 

\section{Good cones}\label{sec:goodcones} 
Here we introduce some basic definitions and terms about convex cones that we will need, mostly following the conventions from toric geometry \cite{Fulton}. 
\begin{definition}
A \emph{convex polyhedral cone} is a set
\be
	C = \{ r_1 m_1 + \ldots + r_n m_n \in \mathbb{R}^r | r_i \geq 0 \}
\ee
generated by any finite set of vectors $m_1,\ldots,m_n \in \mathbb{R}^r$. 
These vectors, or sometimes the corresponding rays, are called the \emph{generators} of the cone.  
\end{definition}
A cone is called \emph{rational} if its generators can be taken in $\mathbb{Z}^r$, and when this is the case we can assume that the generators are primitive, i.e. that $\gcd(m_i^1,\ldots,m_i^r)=1$. 
A cone is called \emph{strongly convex} if
\[
	C \cap (-C) = \{ 0 \},
\]
and we consider only cones of this type. 
We also need the concept of the dual cone $\check C$,
\[
\check C = \{ u \in \mathbb{R}^r | u\cdot v \geq 0 \ \forall v \in C \}. 
\] 
A \emph{face} $f$ of $C$ is the intersection of $C$ with any supporting hyperplane, $f = C \cap u^{\perp} = \{ v \in C | u\cdot v = 0 \}$ for some $u \in \check C$.
All faces of $C$ will themselves be convex polyhedral cones. 
A co-dimension 1 face of of the cone is called a \emph{facet}, and a dimension 1 face is called a generator. 
We can also describe a cone in a dual description as the intersection of half-spaces, i.e. as the set
\be
	C = \{ x \in \mathbb{R}^r | x\cdot v_i \geq 0, \ i = 1, \ldots, n \}
\ee
for the set of inwards normals $v_1,\ldots, v_n$ associated with the facets of $C$. 
To each co-dimension $k$ face, there will be $k$ associated normal vectors.
We note that the dual cone $\check C$ is generated by the set of normals of $C$, and that its set of normals will be the generators of $C$. 
We also write $C^\circ$ for the interior of a cone, i.e. $C^\circ = C - \partial C$. 
 \begin{remark}
 From now on, when we  write cone we mean a strongly convex rational cone. 
 \end{remark}

Further, we are interested in cones with a particular property, namely that of being \emph{good}, which is defined as:
\begin{definition}
	A cone $C$ of dimension $r$ is \emph{good} if at every codimension $k$ face, the associated $k$ normals $v_{i_1},\ldots,v_{i_k}$ satisfy 
	\be
		\mathrm{Span}_{\mathbb{R}} \bra v_{i_k},\ldots,v_{i_k} \ket \cap \mathbb{Z}^n =  \mathrm{Span}_{\mathbb{Z}} \bra v_{i_1},\ldots, v_{i_k} \ket .
	\ee
\end{definition}
It is easy to see that this condition is equivalent to requiring that at every codimension $k$ face, $[ v_{i_1},\ldots, v_{i_k} ]$ regarded as an $k\times r$ matrix, can be completed into an $SL_r (\mathbb{Z})$ matrix.
This goodness condition was introduced by Lerman \cite{Lerman} in the context of contact toric geometry, where the cone appears as the image of the moment map of the torus action on a toric space, and the goodness condition corresponds to the toric space being smooth. 
So sometimes these cones are referred to as smooth cones.  
From the point of view that the functions we study have a connection to toric manifolds, it's a natural condition to impose, and as we will see in the next subsection, the goodness condition is what lets us relate cones and modular properties.

We will also use the notion of cones being \emph{standard}, by which we mean that the cone is simplicial (i.e. the number of generators is equal to the dimension of the cone) and that its normals (or generators) form an $SL_r(\mathbb{Z})$ element, implying that up to an $SL_r(\mathbb{Z})$ transformation the cone is the same as $\mathbb{R}^r_{\geq 0}$.

\subsection{Good cones, $SL_r (\mathbb{Z})$ and modular transformations}\label{sec:conemodularity}
Let $C$ be a good cone of dimension $r>2$. 
From the goodness condition, every codimension $k$ face of the cone give rise to $SL_r(\mathbb{Z})$ elements. 
In particular we care about the $SL_r (\mathbb{Z})$ elements associated to the generating rays of $C$, so let $\rho$ be such a ray generated by the vector $m$. 
Associated to $\rho$ we have a set of $r-1$ normal vectors, $v_1^\rho, \ldots, v_{r-1}^\rho$.
We choose an ordering of these normals such that 
\[
	\det [m , v_1^\rho, \ldots, v^\rho_{r-1} ] > 0 ,
\]
which exists when the dimension of the cone is larger than 2 (in which case we have no choice of ordering). 
The goodness condition then guarantees that there exist an integer vector $n^\rho\in\mathbb{Z}^r$ such that 
\[
	\det [ n^\rho ,v_1^\rho, \ldots, v^\rho_{r-1} ] = 1,
\]
where we keep the ordering of the $v^\rho_i$'s from above. 
The choice of $n^\rho$ is not unique, there will be a family of solutions parametrized by $r-1$ integers, corresponding to shifting $n^\rho$ by linear combinations of the associated normal vectors.  
So in this way we define our $SL_r(\mathbb{Z})$ element associated to a generator $\rho$ as 
\be \label{eq:Kfdef1}
\tilde K_\rho = [ n^\rho ,v_1^\rho, \ldots, v^\rho_{r-1} ]^{-1} , 
\ee
which can be thought of as mapping this ``corner'' of the cone into the standard cone $\mathbb{R}^r_{\geq 0}$. 

An alternative way of seeing why this choice of ordering is appropriate is to realize that with it, the first row of $\tilde K_\rho$ will  be equal to $m$, independently of the choice of $n^\rho$. 
The rest of the rows of $\tilde K_\rho$ can be thought of as the other generators of the cone described by the inward normals $\{ n^\rho, v_1^\rho , \ldots, v^\rho_{r-1} \}$, and they will depend on the choice of $n^\rho$ (since all other generators of this cone will have $n^f$ associated to it).

The dimension $r=2$ is different: here generators and normals are in one-to-one correspondence, so there is no freedom of choosing an appropriate ordering if we want the added line described by $n^\rho$ to be first in our matrix. 
This means that we will not be guaranteed to get an $SL_2(\mathbb{Z})$ matrix, but rather a $GL_2(\mathbb{Z})$ element (i.e. we allow the determinant to be $\pm 1$ instead of only $+1$). 
Otherwise, nothing changes and in the discussion below, for $r=2$ one just substitutes $SL$ with $GL$ everywhere.

For later purposes it is natural to embed $\tilde K_\rho$ trivially into $SL_{r+1}(\mathbb{Z})$, by defining 
\be
	K_\rho = \begin{pmatrix}
		\tilde K_\rho & 0 \\
		0 & 1 
	\end{pmatrix}. 
\ee
$K_\rho$ then acts naturally as a modular (fractional linear) transformation on a set of parameters $(z|\tau_1,\cdots,\tau_{r})$.
Explicitly we consider $\underline \tau = (\tau_1,\cdots,\tau_{r})$ as homogenous coordinates in $\mathbb{P}^r$, i.e. we extend them into $\tilde \tau = (\underline\tau, 1)$ and then an element $g \in SL_{r+1}(\mathbb{Z})$ acts on $(z| \underline \tau)$ as
\be \label{eq:groupaction1}
	g\cdot (z | \underline \tau) = \left ( \frac{z}{(g\tilde \tau)_{r+1} } | \frac{(g\tilde\tau)_1}{(g\tilde \tau)_{r+1} }, \cdots, \frac{(g\tilde\tau)_r}{(g\tilde \tau)_{r+1} } \right ).
\ee
This group action, combined with one additional element called $S$ defined below, with group elements $K_\rho$ defined from the cone, will describe the modular properties of the generalized multiple gamma functions.
The element $S \in SL_{r+1}(\mathbb{Z})$ that we use is given by, for $r>1$: 
\be \label{eq:Sdefinition}
	S = \begin{pmatrix}
	0 & 0 & -1 \\ 
	0 & \mathds{1}_{r-1} & 0 \\
	1 & 0 & 0 
	\end{pmatrix}, \ \ \ \ 
	S^{-1} =  \begin{pmatrix}
	0 & \ldots & 1 \\ 
	0 & \mathds{1}_{r-1} & 0 \\
	-1 & 0 & 0 
	\end{pmatrix},
\ee
which is to be thought of as some sort of ``S-duality'' element, generalizing the usual $S$-element of $SL_2(\mathbb{Z})$.  

\subsection{Cone subdivision}
In the next section we will discuss sums over lattice points inside cones, and for this we need to know that a certain useful subdivision of the cone exists. 
To be precise, we need to be able to subdivide a given $r$-dimensional cone into simplicial standard cones. 
This is a fairly standard result from the study of toric varieties, where the existence of such a subdivision corresponds to the existence of a resolution of singularities, proven for example in chapter 2 of \cite{Fulton}.
For completeness we present the proof here, as well as a slight further result that we will also need. 
\begin{proposition} \label{prop:subdivision}
For any cone $C$ there exists a subdivision of $C$ into $\cup_i C_i$ where $C_i$ are standard cones. 
Further, the subdivision can be chosen so that the sub-cones form a simplicial triangulation of the base of the cone.   
\end{proposition}
\begin{proof}
As a first step, we can subdivide $C$ into a collection of simplicial cones by subsequently adding more vectors to our set of generators, or through barycentric subdivision. 
This gives us a subdivision of $C$ into simplicial cones, and next we show that each of these can be subdivided into standard cones in the above sense.
So consider a simplicial cone $\Delta$ with generators $\{ m_1, \ldots, m_r\}$, and define its multiplicity as $\det [ m_1, \ldots, m_{r} ] \geq 1$.
If the multiplicity is 1, then the cone is standard and we are done, so assume that it is greater than 1.
Then we consider a lattice point inside $\Delta$ of the form 
\[
	 u = q_1 m_1 + \cdots + q_r m_r , \ \ q_i \in \mathbb{Q},  \ \ 0 \leq q_i < 1 . 
\]
We can assume that $u$ is primitive, and then 
\[
	\det [ m_1 , \ldots, m_{i-1} , u , m_{i+1} , \ldots, m_r ] = \det [ m_1 , \ldots, m_{i-1}, m_i , m_{i+1} , \ldots , m_r ] q_i \in \mathbb{Z} ,
\]
where the result is still an integer since $u$ is a lattice point. 
Now since $0\leq q_i < 1$, we see that the simplicial cones generated by $\{ m_1,\ldots, m_{i-1},u,m_{i+1},\ldots, m_r \}$ have smaller multiplicities than the original cone $\Delta$ for all values of $i$, which gives us our subdivision algorithm.
It's clear that the algorithm will terminate in a finite number of steps, giving us standard cones in the end. 
This proves the first statement of the proposition. 

The case that might not give us a simplicial triangulation of the base, is the case when some $q_i = 0$. 
In this case the new generator  $u$ is located inside a face $F$ of $\Delta$, which might lead to a non-simplicial division.
To avoid this, we modify the procedure so that when this happens, we subdivide not only $\Delta$, but also all cones that share the face $F$. 
This subdivision will decrease the multiplicity of the other cones as well, as seen from the above argument, so we again have a good subdivision algorithm, and it will give us a simplicial subdivision of $C$.
\end{proof}

We will also need the following lemma, which is central in our proofs of the infinite product representation of the multiple sine functions, and factorization property of the generalized multiple elliptic gamma function. 

\begin{lemma} \label{lem:rayexclusion}
Let $C$ be a $r$-dimensional cone, and let $\rho$ be a ray in $C$ not equal to any of the generating rays of $C$. 
Then we can find a simplicial subdivision of $C$ such that $\rho$ is not a generating ray of any of the cones in the subdivision. 
\end{lemma}
\begin{proof}
For any such ray $\rho$, we start by selecting a good simplicial $r$-dimensional cone $C_\rho \subset C$ that contains $\rho$ but doesn't have it as one of its generators. 
See figure \ref{fig:rayexclusion} for an illustrative example when $r=3$. 
If $\rho$ is strictly inside $C$, i.e. not contained in any of its proper faces, we can select $C_\rho$ such that it's also strictly inside it. 
Otherwise if $\rho$ is contained a face of $C$, it obviously also need to be in a face of $C_\rho$. 
In any case, we proceed by subdividing $C - C_\rho$: a set that is not a proper cone, but that can be subdivided into a set of cones.
Each of these cones will then be further subdivided according to \ref{prop:subdivision} until we have a simplicial subdivision where all cones are standard. 
It's clear that the procedure, which might involve further subdivisions of $C_\rho$ as well as the other cones, won't ever introduce any generating ray along $\rho$, since the only new generators introduced into $C_\rho$ will lie along its faces shared with the other cones in our subdivision: and $\rho$ is either strictly inside $C_\rho$, or along a face of $C$, which isn't shared with other cones of our subdivision.  
\end{proof}

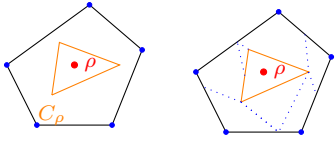
\begin{figure}
\begin{center}
\begin{tikzpicture}
\draw (0,0) -- (1,0) -- (1.4, 1) -- (0.7, 1.6) -- (-0.4, 0.8) -- (0,0);
\draw [blue,fill](-0.4, 0.8) circle (0.03);
\draw [blue,fill](0,0) circle (0.03);
\draw [blue,fill](1,0) circle (0.03);
\draw [blue,fill](1.4, 1) circle (0.03);
\draw [blue,fill](0.7, 1.6) circle (0.03);
\draw [red,fill](0.5, 0.8) circle (0.04) node[draw=none,fill=none,font=\scriptsize,right] {$\rho$};
\draw [orange](0.2, 0.4) node[draw=none,fill=none,font=\scriptsize,below] {$C_\rho$}-- (1.1, 0.8) -- ( 0.3, 1.1) -- (0.2, 0.4); 

\end{tikzpicture}
\ \ \ \ 
\begin{tikzpicture}
\draw (0,0) -- (1,0) -- (1.4, 1) -- (0.7, 1.6) -- (-0.4, 0.8) -- (0,0);
\draw [blue,fill](-0.4, 0.8) circle (0.03);
\draw [blue,fill](0,0) circle (0.03);
\draw [blue,fill](1,0) circle (0.03);
\draw [blue,fill](1.4, 1) circle (0.03);
\draw [blue,fill](0.7, 1.6) circle (0.03);
\draw [red,fill](0.5, 0.8) circle (0.04) node[draw=none,fill=none,font=\scriptsize,right] {$\rho$};
\draw [orange](0.2, 0.4) -- (1.1, 0.8) -- ( 0.3, 1.1) -- (0.2, 0.4); 
\draw [blue, dotted] (0.2, 0.4) -- (0.7, 0.0) -- (1.1, 0.8) -- (1.05, 1.3);
\draw [blue, dotted] (1.1, 0.8) -- (1.2, 0.5);
\draw [blue, dotted] (0.2, 0.4) -- (0.7, 0.0);
\draw [blue, dotted] (0.2, 0.4) -- (-.3, 0.6);
\draw [blue, dotted] (0.25, 0.75) -- (0.15, 1.2);
\draw [blue, dotted] (0.3, 1.1) -- (0.15, 1.2);
\end{tikzpicture}
\end{center}
\caption{Both figures show the polygon base of a 3d cone $C$, where in the first one we show a given ray $\rho$ represented by a point in the base, and a selected simplicial standard cone $C_\rho$ that surrounds it. 
In the next figure, we show some arbitrary subdivision of $C-C_\rho$ into convex good cones, not necessarily simplicial. 
Finally we subdivide each of these into standard cones; something not shown in this figure. }\label{fig:rayexclusion}
\end{figure}

\subsection{Summing over cones}\label{sec:conesums}
In this section, we describe some technology that we need in order to define our generalized functions.
Let $C$ be a good cone of dimension $r$, and let $\re \underline \go$ be in its dual cone.
We wish to consider for example the following sum
\be
	\sum_{n \in C \cap \mathbb{Z}^r } e^{- n \cdot \underline \go },
\ee	
and we see that it converges absolutely. 
However we want to analytically continue it so that it can be made sense of also for generic $\underline \go$ (away from some critical locus at least). 
We do this  using a simplicial subdivision of the cone to perform the sum, and from there we can define our analytical continuation.
Since the original sum is absolutely convergent, we are free to switch order of summation, which is what a simplicial subdivision amounts to, so the analytically continued result will not depend on the chosen subdivision. 
We know from our proposition \ref{prop:subdivision} that a simplicial subdivision of $C$ exists, so we let $C = \cup_i C^r_i$ be such a subdivision, where $\{ C^r_i \}$ is a collection of standard simplicial $r$-dimensional cones, and write the sum as
\[ 
	\sum_{n \in C \cap \mathbb{Z}^r } e^{- n \cdot \underline \go } = \sum_i \sum_{n\in C^r_i \cap \mathbb{Z}^r } e^{- n \cdot \underline \go }  - \{ \mathrm{overcounting} \} , 
\]	
where we realize that we have to remove the terms that are being counted twice in the first sum. 
These points are the ones that are shared among exactly two interior faces of our subdivision, so to correct our counting we need to remove them once.
These interior faces will be some collection of $r-1$ dimensional cones, call them $\{ C^{r-1}_j \}$, and we need to remove their contributions once. 
Doing this means that the faces shared among these cones will be under-counted: at a place where $n$ number of $r-1$ cones meet, $n$ of the $C^r_i$ cones will also have met, and thus we add and remove the points along these $r-2$ dimensional faces the same number of times, so we have to add them back once. 
This is true only when the cones overlap inside $C^\circ$, if they only overlap on the boundary we do not have to add anything, since when say $n$ $C^r_i$-cones overlap at the boundary, only $(n-1)$ of the $(r-1)$-dimensional cones going into the interior of $C$ will overlap: the others will be along the boundary of $C$. 
So the rule is that we only add and remove points inside cones that are not contained in the boundary of $C$. 
The pattern of adding and removing points along shared faces continues until we are left with only the $0$-dimensional face (i.e. vertex) that is shared by all the cones, namely the origin, and we can write the sum as
\be \label{eq:conesubdiv1}
	\sum_{n \in C \cap \mathbb{Z}^r } e^{- n \cdot \underline \go } 
	= \sum_{k=1}^r \sum_{i=1}^{N_k} \sum_{n\in C^{k}_i \cap \mathbb{Z}^r } (-1)^{r-k} e^{- n \cdot \underline \go }  ,
\ee
where $N_k$ is the number of cones of dimension $k$ in our subdivision. 
This is essentially nothing but the usual inclusion-exclusion principle that we apply in order to rearrange our sum into a more useful form. 

Let's show explicitly that this sum counts every lattice point in our original cone exactly once. 
So consider an arbitrary point (not equal to the origin, which we will consider separately) $p \in C$.
The simplicial subdivision of $C$ gives us a simplicial subdivision of the base polytope of the cone, which is a simplicial complex. 
Consider now the sub-complex of this, which includes only the simplices whose corresponding cones contains the point $p$, call this $\Delta_p$. 
This complex is an intersection of a number of starshaped sub-complexes, and thus a standard result tells us that it's contractible. 
However, when counting the number of times $p$ appears in our sum, we don't wish to count the boundaries of $\Delta_p$, since they are not contributing to the number of times we count $p$ in our sum. 
Thus, we consider $\Delta_p / \partial \Delta_p$, and compute its relative Euler character:
\be
	\chi(\Delta_p, \partial \Delta_p ) = \chi(\Delta_p) - \chi(\partial \Delta_p) = 1 - ( 1 - (-1)^r ) = (-1)^r 
\ee
using that $\Delta_p$ is homotopic to a point and that $\partial \Delta_p$ is homotopic to $S^{r-2}$. 
This is directly related to the contribution to our sum from the point $p$:
\[
	e^{- p\cdot \underline \go } \sum_k (-1)^{r-k} E_{k}^{\Delta_p/\partial \Delta_p } = e^{-p \cdot \underline \go } (-1)^r \chi(\Delta_p,\partial \Delta_p ) = e^{-p\cdot \underline \go } ,
\]
proving that we count each point exactly once. 

We need to consider the origin separately; but in a very similar fashion. 
The origin is shared among all the simplices, so here we are instead considering the entire complex $\Delta$.
But again in our sum, we need to exclude its boundary $\partial \Delta$; so we consider $\Delta/\partial\Delta$.
This gives the relative Euler character
\be
	\chi ( \Delta, \partial \Delta ) =  \chi(\Delta) - \chi(\partial \Delta ) = 1 - (1 - (-1)^r ) = (-1)^r . 
\ee
So again the same story as above repeats itself and we get that the origin also is counted only once.

Now because of our subdivision, each of the $C_i^r$-cones will be standard, meaning that we can find $SL_r (\mathbb{Z})$ transformations sending them to the cone $\mathbb{R}^r_{\geq 0}$. 
The same can be done for all the lower-dimensional faces as well, by further subdivisions if necessary. 
Denote by $A_i^k$ the $SL_r (\mathbb{Z})$ element that sends the cone $C_i^k$ to the standard cone $\mathbb{R}^k_{\geq 0}$. 
When acting on a lower-dimensional cone we choose this matrix such that it maps the $k$ normals into the first $k$ basis vectors of $\mathbb{R}^r$, which is what we mean by it being mapped to $\mathbb{R}^k_{\geq 0}$. 
With this, we can compute
\be \label{eq:conesubdiv2}
\begin{split}
	\sum_{k=1}^r \sum_{i=1}^{N_k} \sum_{n\in C^{k}_i \cap \mathbb{Z}^r } (-1)^{r-k} e^{- n \cdot \underline \go }  
	= \sum_{k=1}^r \sum_{i=1}^{N_k} \sum_{(A^k_i n) = m \in \mathbb{Z}^{k}_{\geq 0} } (-1)^{r-k} e^{ - m \cdot (A_i^k \underline \go ) } \\
	=  \sum_{k=1}^r \sum_{i=1}^{N_k} \frac{(-1)^{r-k}} { \prod_{j=1}^k ( 1 - e^{ - (A_i^k \underline \go)_j } ) } . 
	\end{split}
\ee
We note here that our original assumption on $\underline \go$ guarantees us that the sums actually all converge, allowing us to write down that final result.  
And this expression for the sum is now what allows us to analytically continue the sum to any generic $\underline \go$, as long as we avoid $\re (A_i^k \underline \go)_j = 0$, i.e. choose $\underline \go$ `generic enough'. 
This particular expression depends on the subdivision chosen, but since so far everything we've done is unambiguous due to the absolute convergence of the sum, the analytical continuation will also be unique, and not depend on the subdivision. 
Because of this, the apparent singularities when $\re (A_i^k \underline \go)_j = 0$ are actually not there, except the ones associated to the generators of $C$ itself, i.e. when $\re (\tilde K_\rho \underline \go)_j = 0$, since these are the only ones we cannot get rid of by changing subdivision. 
\emph{For the rest of this article, when we write a sum of this form over a cone $C$, we understand it implicitly as this analytical continuation defined above. }

For our purposes, we also need to consider the sum over lattice points strictly inside $C$, i.e. in $C^\circ$
\be
	\sum_{n\in C^\circ \cap \mathbb{Z}^r } e^{ - n\cdot\underline \go}. 
\ee
Now, if we follow the same steps as above, but instead keep summing over the internal points of the subdivided cones, the shared faces and so on, then we realize that to properly count all points in the interior, we only have to add terms, never subtract them, so we do not get the factor $(-1)^{r-k}$ that we had above. 
We also do not have to think about the origin, as it isn't included in any of our cones, and we get  
\[
	\sum_{k=1}^r \sum_{i=1}^{N_k} \sum_{n\in (C^{k}_i)^\circ \cap \mathbb{Z}^r } e^{- n \cdot \underline \go }  
	= \sum_{k=1}^r \sum_{i=1}^{N_k} \sum_{(A^k_i n) = m \in \mathbb{Z}^{k}_{> 0} } e^{ - m \cdot (A_i^k \underline \go ) } 
	=  \sum_{k=1}^r \sum_{i=1}^{N_k} \prod_{j=1}^k \frac{e^{- (A_i^k \underline \go )_j } } { 1 - e^{ - (A_i^k \underline \go)_j } } . 
\]
This can be rewritten as 
\be
	 \sum_{k=1}^r \sum_{i=1}^{N_k} \prod_{j=1}^k \frac{1 } {  e^{  (A_i^k \underline \go)_j }- 1 }
	 = (-1)^r  \sum_{k=1}^r \sum_{i=1}^{N_k}  \frac{(-1)^{r-k} } { \prod_{j=1}^k (1 - e^{ (A_i^k \underline \go)_j } ) } 
	 = (-1)^r \sum_{n \in C\cap\mathbb{Z}^r } e^{ n \cdot \underline \go },
 \ee
where we in the final step use our analytic continuation to make sense of this sum. 
We summarize this in a lemma: 
\begin{lemma}\label{lem:conesum1}
Let $C$ be a good cone in $r$ dimensions, and let $\underline \go \in \comp^r$ be such that $\re (\tilde K_\rho \underline \go )_j \neq 0$ for all generators $\rho$ of $C$, then 
\be
	\sum_{n \in C^{\circ} \cap \mathbb{Z}^r } e^{- n\cdot \underline \go}  = (-1)^r \sum_{n \in C \cap \mathbb{Z}^r } e^{ n \cdot \underline \go }, 
\ee
where the sums are understood through the analytic continuation outlined above. 
\end{lemma}

\section{Generalized multiple sine functions}

\subsection{Definition of the generalized multiple sine functions}
Let $C$ be a good cone of dimension $r$ and 
and assume that $\underline \go \in \comp^r$ is such that there exists a $k\in\comp - \{ 0 \}$ such that $\re ( k \underline \go ) \in (\check C)^{\circ}$. 
In this case we define the generalized multiple zeta function associated with $C$, $\zeta^C_r$, as 
\be
	\zeta^C_r (s, z | \underline \go ) = \sum_{n \in C \cap \mathbb{Z}^r } \frac{1}{( z + n \cdot \underline \go )^s },
\ee
which is absolutely convergent for $\re s>r$, and where the exponential is rendered one-valued. 
This function is analytically continued to $s \in \comp$, and thus we can use it to define the generalized multiple gamma function associated to $C$ as 
\be
	\Gamma_r^C ( z | \underline \go ) = \exp \left ( \frac{\partial}{\partial s} \zeta^C_r(s,z|\underline \go ) |_{s=0} \right ) . 
\ee
This mirrors the definitions of the ordinary zeta and gamma functions; it's again exactly the procedure of zeta-regularizing the infinite product over the lattice points in the cone. 
We can repeat the construction for $C^\circ$, simply replacing the sum in $\zeta^C_r$ with one over $C^\circ$, and with this we can define the generalized multiple sine function.
\begin{definition}
The generalized multiple sine function associated to $C$ is given by 
\be
	S_r^C(z|\underline \go) = \Gamma^C_r ( z | \underline \go)^{-1} \Gamma^{C^\circ}_r ( -z | \underline \go )^{(-1)^r } .  
\ee
\end{definition}
If $C = \mathbb{R}^r_{\geq 0}$ this is equal to the usual multiple sine function.

\subsection{Generalized Bernoulli polynomials}
Before giving the integral representation of these generalized multiple sine functions, we introduce the \emph{generalized Bernoulli polynomials} associated to a cone, which are defined through the formal generating series
\be \label{eq:genBernoullidef}
	(-1)^r t^r e^{zt} \sum_{n\in C \cap\mathbb{Z}^r} e^{(\underline \omega \cdot n) t } = \sum_{n=0}^\infty B_{r,n}^C ( z | \underline \omega ) \frac{t^n}{n!}. 
\ee
This generalizes the usual definition of Bernoulli polynomials and reduces to it if $C= \mathbb{R}^r_{\geq 0}$.
These polynomials will show up in the integral and infinite product representations of the generalized multiple sine functions, as well as in the modularity properties of the generalized multiple elliptic gamma functions. 
These can be seen to share some (but not all) of the functional relations the original Bernoulli polynomials enjoy, for example it's easy to show that both of the properties shown in equation \eqref{eq:bernoulliproperties} also apply to $B_{r,n}^C$ with the appropriate modification to the second one:
\be \label{eq:bernoulliCproperty}
	B_{r,n}^{C^\circ} ( - z | \underline \go ) = (-1)^n B_{r,n}^C ( z | \underline \go ),
\ee
where we in the proof need to apply lemma \ref{lem:conesum1}. 

We also remark (without giving proof, since it's not important for the main topic of this note) that the coefficients of the polynomials $B_{r,r}^C$ can be given in terms of geometrical data of the cone $C$. 
For example, the coefficient of the leading term ($z^r$) in $B_{r,r}^C$ is proportional to the volume of $C$ capped by the plane $(\re \underline \go ) \cdot y = 1$ (when $\re \underline \go \in \check C$), and the next-to-leading coefficient is proportional to the surface area of the same capped cone and so on. 
We also mention that we've found some intriguing relations between the constant terms of these Bernoulli polynomials and Dedekind sums, the usual ones when $r=2$ and higher-dimensional versions for higher $r$. 
However at present we don't have a good enough understanding of this to write something sensible (outside $r=2,3$), so we leave it for possible future work.

\subsection{Integral representation}
In this section, we derive an integral representation of our generalized multiple sine functions, which we will then use to derive an interesting infinite product representation. 
We start by giving an integral representation of our generalized zeta function. 
Assume again that $\underline \go$ is such that there exists a $k\in \comp - \{ 0 \}$ such that $k\underline \go \in \check C^\circ$ and  that $\re ( k z ) > 0 $. 
Following the argument of Barnes \cite{Barnes}, one can see that under these conditions on $\underline \go$, $\zeta^C_r$ has the following integral representation:
\be
	\zeta^C_r ( s, z | \underline \go ) = - \frac{\Gamma (1-s)} { 2 \pi i  } \int_L (-t)^{s-1} \sum_{n \in C \cap \mathbb{Z}^r } e^{- ( z + m \cdot \underline \go ) t } dt ,
\ee
where $L$ is a contour enveloping the half-line through the origin in the direction of $k$, see figure \ref{fig:Lcont}. 
We can also repeat that every time we write a sum like this over lattice points in a cone $C$, we understand it the way described in section \ref{sec:conesums}, i.e. as an analytical function of $\underline \go$ that can be defined by a simplicial subdivision of $C$. 
This choice of contour guarantees us that the integral will converge, and that the only pole enveloped by the contour will be the one at the origin. 
\begin{figure}
\begin{center}
\begin{tikzpicture}
\draw[thin, gray, ->] (-0.6,0) -- (3,0);
\draw[thin, gray, ->] (0,-0.4) -- (0,1.7);
\draw (0,0) -- (3,1.5) node[anchor=south east] {$L$};
\draw [->] (3,1.7) -- (1.5, 0.95);
\draw (1.5, 0.95) -- (-.08, 0.16);  
\draw[->] (0.08, -0.16) -- (1.5, 0.55);
\draw (1.5, 0.55) -- (3, 1.3);
\draw (-0.08, 0.16) arc (116.6:296.6:0.179);
\fill[black] (0,0) circle (0.05cm);
\end{tikzpicture}
%
%
\end{center}
\caption{The contour of integration enveloping the line $L$.}\label{fig:Lcont}
\end{figure}
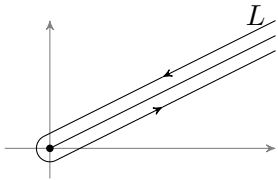
We then take a derivative of this expression with respect to $s$ and use that to find an integral representation of the multiple gamma function $\Gamma_r^C$: 
\be
	\Gamma^C_r ( z | \underline \go ) = \exp \left [  \frac{1}{2\pi i } \int_L \frac{e^{-zt}}{t} (\log (-t) + \gamma ) \sum_{n \in C\cap\mathbb{Z}^r } e^{- (n\cdot \underline \go ) t } dt \right ]  ,
\ee
where $\gamma$ is Eulers constant, appearing from the derivative of the usual gamma function, and the logarithm is chosen with a branch cut along the half-line enveloped by $L$ and such that $\log(-t)$ is real when $t$ is real and negative. 

This gives us an integral representation of $S_r^C$, which we use to find the following properties, that closely mirrors properties of the usual multiple sine. 
\begin{proposition}\label{prop:SrCintegral}
(i) The generalized multiple sine function satisfies	$S_r^C ( c z | c \underline \go) = S_r^C ( z | \underline \go )$ for $c\in \comp - \{ 0 \}$.\\
(ii) When $\underline \go$ is such that $Im \frac{(\tilde K_\rho \underline \go)_j}{(\tilde K_\rho \underline \go )_1 } \neq 0$ for all generating rays $\rho$ of $C$ and for $j=2,\ldots,r$, then
$S_r^C$ has the integral representations 
\bea
	S_r^C ( z | \underline \go ) &=&  \exp \left [ (-1)^r \frac{i \pi }{r!} B^C_{r,r}(z | \underline \go) + \int_{\mathbb{R}+i0} \frac{e^{zt}}{t} \sum_{n\in C\cap\mathbb{Z}^r } e^{(n\cdot \underline \go )t } dt \right ] \\
						&=&  \exp \left [ (-1)^{r+1}  \frac{i \pi }{r!} B^C_{r,r}(z | \underline \go) + \int_{\mathbb{R}-i0} \frac{e^{zt}}{t} \sum_{n\in C\cap\mathbb{Z}^r } e^{(n\cdot \underline \go )t } dt \right ] \ . 
\eea 
\end{proposition}
\begin{proof}
We first prove (i). Assume that $\underline \go$ and $z$ are such that we can apply the integral representation of our generalized zeta function. 
Given this, we study $\Gamma_r^C ( c z | c\underline \go)$ for $c \in C-\{0 \}$ with $-c \notin L$. 
Let $c^{-1} L$ denote the half-line along $c^{-1} k$,  and call the contour which envelopes it $c^{-1} \tilde L$. 
Then 
\[
	\Gamma_r^C ( cz | c\underline \go ) = \exp \left [ \frac 1 {2\pi i } \int_{c^{-1} \tilde L } \frac{e^{-czt} }{t} (\log(-t) + \gamma ) \sum_{n \in C \cap \mathbb{Z}^r } e^{- c( n\cdot \underline \go ) t } dt \right ] , 
\]
where now the cross cut of the $\log$ is along $c^{-1} L$.
 We again remind the reader that the sum over lattice points in $C$ should be thought of as explained in \ref{sec:conesums}.
If one then changes $t$ into $c^{-1} t$ and changes the branch of the logarithm, we get 
\[
	\Gamma_r^C ( cz | c\underline \go ) = \exp \left [ \frac 1 {2\pi i } \int_{ \tilde L } \frac{e^{-zt} }{t} (\log(-t) - \log c + \gamma ) \sum_{n \in C \cap \mathbb{Z}^r } e^{- ( n\cdot \underline \go ) t } dt \right ] . 
\]
Next we define for notational convenience the function 
\[
	\phi^C ( t ) = \frac{e^{zt}}{t} \sum_{n\in C\cap \mathbb{Z}^r } e^{(n\cdot\underline \omega ) t },
\]
in terms of which our gamma function can be written 
\[
	\Gamma_r^C ( cz | c\underline \go ) = \exp \left [- \frac 1 {2\pi i } \int_{ \tilde L } \phi^C (-t) (\log(-t) - \log c + \gamma )  dt \right ] . 
\]
Next, we consider $\Gamma_r^{C^{\circ}}(- c z | c \underline \go )$. 
The integral representation of this is, following the same steps as above,
\[
	\Gamma_r^{C^\circ } ( - cz | c\underline \go ) = \exp \left [ \frac 1 {2\pi i } \int_{ \tilde L } \frac{e^{zt} }{t} (\log(-t) - \log c + \gamma ) \sum_{n \in C^\circ \cap \mathbb{Z}^r } e^{- ( n\cdot \underline \go ) t } dt \right ] . 
\]
Now, using lemma \ref{lem:conesum1} we can rewrite part of the integrand here as follows: 
\[
	\frac{e^{zt}}{t} \sum_{n \in C^\circ \cap \mathbb{Z}^r } e^{- ( n\cdot \underline \go ) t } = \frac{e^{zt}}{t} (-1)^r \sum_{n \in C \cap \mathbb{Z}^r } e^{( n\cdot \underline \go ) t } = (-1)^{r} \phi^C ( t ) ,
\]
which lets us write the generalized multiple sine as 
\[
	S_r^C ( c z | c \underline \go ) = \exp \left [ \frac 1 {2\pi i } \int_{ \tilde L } ( \phi^C (-t) +  \phi^C ( t ) ) (\log(-t) - \log c + \gamma )  dt \right ] .
\]
Now we note that the integral $\int_{\tilde L} ( \phi^C (t) + \phi^C ( -t ) ) dt $ can be replaced by an integrand along a contour $C_0$ around the origin, since the origin is the only pole inside this contour and both $\phi^C (t)$ and $\phi^C(-t)$ are rapidly decreasing along $\tilde L$.
This rapid decrease is since either the pre factor $e^{\pm zt}/t$ is exponentially decreasing and the sum is tending to some constant, or in the case that $e^{\pm zt }$ is growing, the sum will be decreasing exponentially at a much quicker rate. 
Further we observe that since $( \phi^C (t) + \phi^C ( -t ) )$ is an even function, its integral around $C_0$ will vanish, allowing us to drop both Euler's constant as well as $\log c $ from the above expression. 
This leaves us with 
\[
	S_r^C ( c z | c \underline \go ) = \exp \left [ \frac 1 {2\pi i } \int_{ \tilde L } ( \phi^C (t) +  \phi^C ( - t ) ) \log(-t)  dt \right ] ,
\]
where the right hand side is independent of $c$. 

This means that $S_r^C(cz | c\underline\go)$ coincides with $S_r^C(z|\underline \go )$ as long as our assumptions on $z$ and $\underline \go$ holds.
But in fact analytic continuation then implies that this is true for any $z$ away from the poles. 

Next, we set $c=1$ and $\underline\go$ and $z$ such that $L=\mathbb{R}_{\geq 0}$, and we integrate along a contour $C_1$ that envelopes the positive real axis, see figure \ref{fig:C1cont}.
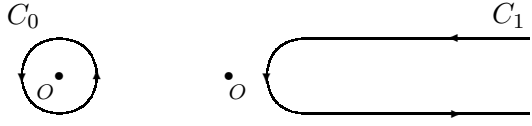
\begin{figure}
\begin{center}
\setlength{\unitlength}{1mm}
\begin{picture}(20,40)(-10,-20)
\put(0,0){\circle*{1}}
\put(-3,-3){$\scriptstyle O$}
\put(-7,7){$C_0$}
\bezier{100}(5,0)(4.7,4.7)(0,5)
\bezier{100}(-5,0)(-4.7,4.7)(0,5)
\bezier{100}(-5,0)(-4.7,-4.7)(0,-5)
\bezier{100}(5,0)(4.7,-4.7)(0,-5)
\put(5,0){\vector(0,1){1}}
\put(-5,0){\vector(0,-1){1}}
\end{picture} \hspace{10mm}
\begin{picture}(40,40)(-10,-20)
\put(-10,0){\circle*{1}}
\put(-10,-3){$\scriptstyle O$}
\put(0,5){\line(1,0){30}}
\put(0,-5){\line(1,0){30}}
\bezier{100}(-5,0)(-4.7,4.7)(0,5)
\bezier{100}(-5,0)(-4.7,-4.7)(0,-5)
\put(20,5){\vector(-1,0){1}}
\put(20,-5){\vector(1,0){1}}
\put(-5,0){\vector(0,-1){1}}
\put(25,7){$C_1$}
\end{picture}
\end{center}
\caption{Two contours of integration, $C_0$ enveloping the origin and $C_1$ enveloping the positive real line excluding the origin. } \label{fig:C1cont}
\end{figure}
The branch cut of the log is chosen along $L$.
Then we stretch the contour of integration along the negative real axis, which implies
\be
	\int_{ C_1 } ( \phi^C (t) +  \phi^C ( - t ) ) \log(-t)  dt = \left (  \int_{-\mathbb{R}+i\epsilon} + \int_{\mathbb{R}-i\epsilon} \right ) ( \phi^C (t) +  \phi^C ( - t ) ) \log(-t)  dt
\ee
Changing $t$ into $-t$ without changing the branch of the logarithm shifts the logarithms by $\pm i\pi$:
\[ \begin{split}
	\int_{-\mathbb{R}+i\epsilon} \phi^C(-t)\log(-t) dt = -\int_{\mathbb{R}-i\epsilon} \phi^C(t) ( \log(-t)-i\pi ) dt \\
	\int_{\mathbb{R}-i\epsilon} \phi^C (-t) \log(-t) dt = - \int_{-\mathbb{R}+i\epsilon} \phi^C(t) (\log(-t) + i\pi ) dt. 
	\end{split}
\]
Using this, we have 
\[
	\left (  \int_{-\mathbb{R}+i\epsilon} + \int_{\mathbb{R}-i\epsilon} \right ) ( \phi^C (t) +  \phi^C ( - t ) ) \log(-t)  dt = 
	- i \pi \int_{-\mathbb{R}+i\epsilon}  \phi^C ( t) dt +  i\pi \int_{\mathbb{R}-i\epsilon} \phi^C ( t) dt 
\]
Finally, we can combine the two integrals into a single one, by switching direction of the first integral and deforming $\mathbb{R}-i\epsilon$ into $\mathbb{R}+i\epsilon$, taking care of the contribution from the pole at the origin:
\[
	- i \pi \int_{-\mathbb{R}+i\epsilon}  \phi^C ( t) dt +  i\pi \int_{\mathbb{R}-i\epsilon} \phi^C ( t) dt  = 2 \pi i  \int_{\mathbb{R} + i \epsilon } \phi^C(t) dt + i \pi \int_{C_0} \phi^C (t) dt .  
\]
Finally, we observe that the residue of $\phi^C(t)$ at $t=0$ is given by our generalized Bernoulli polynomial, or more precisely 
\[
\frac 1 2 \int_{C_0} \phi^C (t) dt = (-1)^r  \frac{\pi i }{r!}B_{r,r}^C ( z | \underline \go ),
\]
as seen from its definition in terms of the formal series in equation \eqref{eq:genBernoullidef}. 
Putting this together, we've shown 
\be
S_r^C ( z | \underline \go ) = \exp \left [ (-1)^r \frac{i \pi }{r!} B^C_{r,r}(z | \underline \go) + \int_{\mathbb{R}+i0} \frac{e^{zt}}{t} \sum_{n\in C\cap\mathbb{Z}^r } e^{(n\cdot \underline \go )t } dt \right ] . 
\ee
The second integral representation follows by choosing to deform the contour into $\mathbb{R}-i0$ in the last step instead. 
\end{proof}

\subsection{Infinite product representation}
In this section we use the integral representation to find an intriguing infinite product representation of $S_r^C$, which lets us write it as a product of shifted q-factorials, again mirroring a property of the usual multiple sine functions, but now with a non-trivial dependence on the choice of cone, with one $q$-factorial factor coming from each generator of the cone.
Before proving this we need the following lemma, that allows us to perform a contour integration.

\begin{lemma}\label{lem:convergence}
Let $C$ be a good cone of dimension $r$, and let $\re \underline \go \in (\check C)^{\circ}$. 
Then there exists a neighborhood $U\subset \comp$ and a real series $\{ a_n \}$ such that 
\[
	\lim_{n\rightarrow \infty} a_n = \infty, \ \ \ \ \mathrm{and} \ \ \lim_{n\rightarrow \infty} \int_{\mathbb{R}+ia_n} \frac{e^{zt}}{t} \sum_{n\in C\cap \mathbb{Z}^r } e^{t (n\cdot \underline \go)} dt = 0,
\]
for $z \in U$.
There also exists another $U'\subset \comp$  and another real series $\{ b_n \}$ such that 
\[
	\lim_{n\rightarrow \infty} b_n = -\infty, \ \ \ \ \mathrm{and} \ \ \lim_{n\rightarrow \infty} \int_{\mathbb{R}+i b_n} \frac{e^{zt}}{t} \sum_{n\in C\cap \mathbb{Z}^r } e^{t (n\cdot \underline \go)} dt = 0,
\]
for $z\in U'$. 
\end{lemma}
\begin{proof}
There exist a small $\epsilon >0$ and a real series $\{ a_n \}$ such that the distances from each pole to any contour $\mathbb{R} + i a_n$ are more than $\epsilon$. 
We now need to argue for the existence of a good set of $z$-values to make the integrals converge. 
We use the subdivision of $C$ described in equation \eqref{eq:conesubdiv2}, and then note that the parameters $ (A^k_i \underline \go )_j$ can be thought of as follows. 
All of them are associated with a particular cone in the subdivision of $C$, and can be seen as the dot-product between a particular generator of this cone and the parameters $\underline \go$. 
And since all these generators lie inside $C$ and we've assumed that $\re \underline \go$ lie strictly inside the dual cone of $C$, all these will have strictly positive real part. 

Then consider the integral in the first statement of the lemma, and use the subdivision of $C$ to break it into a sum of integrals, each of which will have the form
\[
	\int_{\mathbb{R}+i a_n} \frac{e^{zt}}{t} \frac{dt}{\prod_{j=1}^k ( 1  - e^{ t (A_{i}^k \underline \go )_j } ) }.
\]
Then by estimating the absolute value of the integrand as $a_n \rightarrow \infty$, we see that it goes to zero if $ 0 < \re z < \re | A_i^k \underline \go |$, $\im z > 0$ and $\im z > \im | A_i^k \underline \go | $. 
We see that such $z$'s will always exist since for all cones in the subdivision, $\re | A_i^k \underline \go |$ is strictly positive, so there will be a smallest one and we can pick $\re z$ smaller than this but still positive.
So there exists a subset $U$ such that the integral vanishes.

Similarly for the second case, we estimate the absolute value of the integrand and find that it goes to zero if $0< \re z < \re |A_i^k \underline \go |$, $\im z < 0$ and $\im z < \im | A^i_k |$, where we again see that a set of such $z$'s exist. 
\end{proof}

Now, we use the integral representation from proposition \ref{prop:SrCintegral}, together with this lemma and our results about sums over cones to prove the following infinite product representation of $S^C_r$. 
\begin{theorem}\label{thm:SrCproductrep}
Let $C$ be a good cone of dimension $r \geq 2$ and let $\underline \go \in \comp^r$ be such that $ \im \frac{ (\tilde K_\rho \underline \go )_j } { (\tilde K_\rho \underline \go )_1 } \neq 0 $ for all generating rays $\rho$ of $C$ and all $j = 2,\ldots,r$. 
Then $S_r^C ( z | \underline \go)$ has the following two infinite product representations
\bea
	S_r^C ( z | \underline \go ) &=& e^{(-1)^r \frac{\pi i }{r!}B_{r,r}^C ( z | \underline \go ) } \prod_{\rho \in \Delta_1^C} ( x_\rho | \underline q_\rho )_\infty \\
		&=& e^{(-1)^{r+1} \frac{\pi i }{r!}B_{r,r}^C ( z | \underline \go ) } \prod_{\rho \in \Delta_1^C} ( x_\rho^{-1} | \underline{ q}_\rho^{-1} )_\infty , 
\eea
where $\Delta^C_1$ is the set of generators of $C$, and 
\[
		x_\rho = e^{2 \pi i \frac{z}{(\tilde K_\rho \go)_1 } }, \hspace{5 mm} 	 \underline q_\rho = \left (e^{2\pi i \frac{(\tilde K_\rho \go)_2}{(\tilde K_\rho \go)_1 } }, \cdots, e^{2\pi i \frac{(\tilde K_\rho \go)_r}{(\tilde K_\rho \go)_1} } \right  ), 
\]
where $\tilde K_\rho$ are the associated $SL_{r}(\mathbb{Z})$ (or, for $r=2$, $GL_2(\mathbb{Z})$) matrices as described in section \ref{sec:conemodularity}.
\end{theorem}
\begin{proof}
We begin from the integral representation of proposition \ref{prop:SrCintegral}, in which the Bernoulli factors are already present, so what we need to do is evaluate the remaining integral. 
We will start from the first of the two integral representations and prove the first factorization formula above in detail, so we  consider the integral
\[
		\int_{\mathbb{R}+i0}  \frac{e^{zt}}{t} \sum_{n\in C\cap \mathbb{Z}^r } e^{t(n\cdot\underline \omega ) } dt .
\]
In order to apply lemma \ref{lem:convergence}, we choose $z$ to lie in the domain where we can close the contour around the upper half-plane, since we are guaranteed that this choice exist. 

Then we can evaluate the integral by summing over all residues in the upper half plane, so let's find out where the poles are located and what their associated residues are. 
We can apply the results of section \ref{sec:conesums} and write the integrand as
\be
	\frac{e^{zt}}{t} \sum_{n\in C\cap \mathbb{Z}^r } e^{t(n\cdot\underline \omega ) } 
	= \frac{e^{zt}}{t} \sum_{k=1}^r \sum_i^{N_k} \frac{(-1)^{r+k} } { \prod_{j=1}^k ( 1 - e^{ t (A_i^k \underline \go )_j } ) } ,
\ee
from which at first glance there seems to be a large number of poles, that might be hard to package. 
However most of these poles are spurious and cancel when we perform the sum, something we can see as follows. 
An apparent family of poles occurs at $t = n / (A^k_i\underline \go)_j$, $n \in \mathbb{Z}$, and as we've explained above, we can think of $(A^k_i\underline \go)_j$ as the dot-product of $\underline \go$ and the $j$:th generator of the cone $C^k_i$ in the subdivision of $C$. 
Since the value of our sum doesn't depend on the choice of subdivision, for any such generating ray $\rho$, we can use lemma \ref{lem:rayexclusion} to choose a subdivision where $\rho$ no longer is a generator.
Thus the associated family of poles are not actually there. 
The only rays we cannot get rid of in this fashion are the generators of $C$ itself: they will of course always be generators in whatever subdivision we choose, and thus they give us the only real families of poles. 
This is the key step of the proof and why we get exactly one contribution from each generator of our cone: all that remain is to sum up all the residues from the poles inside our contour.

%

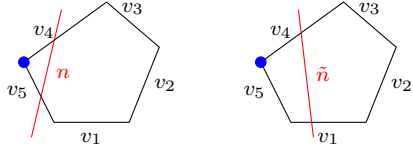
\begin{figure}
\begin{center}
\begin{tikzpicture}
\draw (0,0) -- (1,0)node[draw=none,fill=none,font=\scriptsize,midway,below] {$v_1$} -- (1.4, 1)node[draw=none,fill=none,font=\scriptsize,midway,right] {$v_2$} -- (0.7, 1.6)node[draw=none,fill=none,font=\scriptsize,midway,above] {$v_3$} -- (-0.4, 0.8)node[draw=none,fill=none,font=\scriptsize,midway,left] {$v_4$} -- (0,0)node[draw=none,fill=none,font=\scriptsize,midway,left] {$v_5$};
\draw [blue,fill](-0.4, 0.8) circle (0.07);
\draw[red] (0.1 , 1.5) -- (-0.3, -0.2)node[draw=none,fill=none,font=\scriptsize,midway,right] {$n$};
\end{tikzpicture}
\ \ \ \ 
\begin{tikzpicture}
\draw (0,0) -- (1,0)node[draw=none,fill=none,font=\scriptsize,midway,below] {$v_1$} -- (1.4, 1)node[draw=none,fill=none,font=\scriptsize,midway,right] {$v_2$} -- (0.7, 1.6)node[draw=none,fill=none,font=\scriptsize,midway,above] {$v_3$} -- (-0.4, 0.8)node[draw=none,fill=none,font=\scriptsize,midway,left] {$v_4$} -- (0,0)node[draw=none,fill=none,font=\scriptsize,midway,left] {$v_5$};
\draw [blue,fill](-0.4, 0.8) circle (0.07);
\draw[red] (0.1 , 1.5) -- (0.3, -0.20)node[draw=none,fill=none,font=\scriptsize,midway,right] {$\tilde n$};;
\end{tikzpicture}
\end{center}
\caption{Both images show the polygon base of a 3d cone, with the associated normal vectors to each face, and with the marked corner ``cut off'' by the plane with normal vector $n \  (\tilde n)$. The left cut is an example of an acceptable cut, and the right one of a bad cut, since the ``cut off'' polygon involves more than one vertex of the original polygon and doesn't give  us a simplicial cone.  }\label{fig:cutting}
\end{figure}

There is a subtlety that we should take care of first however, namely that in a general subdivision we can have a number of cones meeting along the generators of $C$, and we would have to add up the contributions from all of them to find the total residues. 
However we can avoid this by choosing a subdivision that ``cuts off'' the corner with the generator of $C$, so that it's only part of a single simplicial cone. 
What we mean by this ``cutting'' might not be immediately apparent, and is perhaps clearest if we think about the base of our cone rather than the cone itself; see figure \ref{fig:cutting} for an illustration. 
The base of the cone is a convex polytope of dimension $r-1$, and a corner (or vertex) $\rho$ of it is associated to $r-1$ normal vectors $v_1^\rho, \ldots, v_{r-1}^\rho $. 
The choice of one additional normal vector $n^\rho$ define a hyperplane cutting through this polytope somewhere, and the claim is that we can pick this vector such that the resulting base polytope of our $r-1$ vectors plus $n^\rho$ gives us a good simplex that doesn't include any other vertex of our base polytope. 
This is always possible due to the goodness condition: there will always exist a choice of $n^\rho$ that completes the $(r-1)$ normals into an $SL_{r}(\mathbb{Z})$ element, and since from the goodness condition we know that the set of normals span the entire lattice, so since we can shift $n^\rho$ by any linear combination of the $v_i^\rho$'s, we see that there is enough freedom of choosing this $n^\rho$ so that we can always make a n acceptable cut, as in figure \ref{fig:cutting}. 
And then the rest of the cone will be a new good cone, which thus can be further subdivided. 
We will now show that the values of the residues coming from the generator of $C$ is independent of this choice of hyperplane. 

So let's consider the contribution of the residues coming from $\rho$.  
As described, we cut it off from the rest of the cone with an appropriate hyperplane with normal $n^\rho$, picked so that 
\[
	\det [ n^\rho, v_1^\rho,\ldots, v_{r-1}^\rho] = 1 ,
\]
where such a $n^\rho$ is guaranteed to exist due to the goodness condition.
We pick an ordering of the normals $v_i^\rho$ so that $m\cdot n^\rho > 0$, where $m$ is the generating vector of $\rho$, so that the matrix $ \tilde K_\rho =  [ n^\rho, v_1^\rho, \ldots, v_{r-1}^\rho]^{-1}$ sends the cone to a properly oriented standard cone. 

Cones in 2d is a special case, as mentioned in section \ref{sec:conemodularity}: there is no ordering to pick and we will get $GL_2(\mathbb{Z})$ matrices instead; but otherwise nothing changes.

With these choices, the family of poles associated to $\rho$ will be at 
\[
	t =  \frac{2\pi i  k} {(\tilde K_\rho \underline \omega )_1}, \ \ \ k \in \mathbb{Z} - \{ 0 \} , 
\]
which is the same as $\frac{2\pi i k}{m \cdot \underline \go}$. 
We now investigate how the residues depend on the choice of hyperplane $n_\rho$. 
Let $n^\rho$ be a vector satisfying the above requirements, and consider another possible choice $\hat n^\rho = n^\rho + m$ that also works. 
The determinant of $[\hat n^\rho, v_1^\rho,\ldots, v_{r-1}^\rho]$ also being equal to one implies that $m$ can be written as a linear combination of $v_i^\rho$'s, 
so that $m=\sum_i a_i v_i^f$, where $a_i \in \mathbb{Z}$. 
The keen reader might here ask why we don't allow $a_i \in \mathbb{Q}$ such that $m$ is still integer: this is because of the goodness condition, which tells us precisely that we can take $a_i$ to be integers without loosing anything. 
Then it follows that
\[
		[\hat n^\rho, v_1^\rho,\ldots, v_{r-1}^\rho] = [ n^\rho , v_1^\rho ,\ldots, v_{r-1}^\rho]
		  \begin{pmatrix}
			1  & 0 & \cdots & 0 \\
			a_1 & 1 & \cdots & 0 \\
			\vdots & 0 & \ddots& \vdots \\
			a_{r-1} & \cdot & \cdots &1
		\end{pmatrix}.
\]
 The inverse of this is
 \[
 	[\hat n^\rho, v_1^\rho,\ldots, v_{r-1}^\rho]^{-1} =  \begin{pmatrix}
			1  & 0 & \cdots & 0 \\
			-a_1 & 1 & \cdots & 0 \\
			\vdots & 0 & \ddots& \vdots \\
			-a_{r-1} & \cdot & \cdots &1
		\end{pmatrix} [ n^\rho , v_1^\rho ,\ldots, v_{r-1}^\rho]^{-1}, 
 \]
from which we see explicitly that the different choices of $n_\rho$ leaves $(\tilde K_\rho \underline\go)_1$ unchanged (as is obvious from how it also is equal to $x_\rho \cdot\underline\go$), whilst the other parameters are shifted by $ -a_i (\tilde K_\rho \underline \go)_1$.
This leaves the residues unchanged since they only depend on the other parameters through $\exp [ 2\pi i \frac{(\tilde K_\rho \underline \go)_j}{(\tilde K_\rho \underline \go)_1}]$.

Now, since we're assuming that $\re \underline\go \in (\check C)^\circ$, we know that $\re (\tilde K_\rho \go )_1 > 0$ for all $\rho$, and the associated poles in the upper half plane are at the points 
\[
	t = \frac{ 2\pi i  n } { (\tilde K_\rho \go )_1 } , \ \ \ n = 1, 2 , \ldots,
\]
and the corresponding residues are given by 
\be
	-\frac{e^{2\pi i n \frac{z}{ (\tilde K_\rho \go )_1} } } { 2\pi i n  \prod_{j=2}^{r} ( 1 - e^{2\pi i n \frac{ (\tilde K_\rho\underline \go )_j }{ (\tilde K_\rho \go )_1 } } )}. 
\ee
We now add up all the residues for one such family, and see that the contribution to the integral is 
\be
 \begin{split}
	&\exp \left [-2\pi i  \sum_{n=1}^\infty \frac{e^{ 2\pi i n \frac{z}{ (\tilde K_\rho\go )_1 } } } { 2\pi i n  \prod_{j=2}^{r} \left ( 1 - e^{2\pi i n \frac{ (\tilde K_\rho\go )_j }{ (\tilde K_\rho\go )_1 } } \right ) } \right ] \\
	&= \exp \left [ - \li_r ( e^{ 2\pi i \frac{z}{ (\tilde K_\rho\go )_1 } } | e^{2\pi i \frac{ (\tilde K_\rho\go )_2 }{ (\tilde K_\rho\go )_1 } } , \ldots,e^{2\pi i \frac{ (\tilde K_\rho\go )_{r} }{ (\tilde K_\rho\go )_1 } } ) \right ] \\
	&=\left ( e^{ 2\pi i  \frac{z}{ (\tilde K_\rho\go )_1 } } | e^{2\pi i  \frac{ (\tilde K_\rho\go )_2 }{ (\tilde K_\rho\go )_1 } } , \ldots,e^{2\pi i  \frac{ (\tilde K_\rho\go )_{r} }{ (\tilde K_\rho\go )_1 } } \right )_\infty , 
\end{split}
\ee
where we use the result of lemma \ref{lem:polylogexp}.
Finally 
we add all the contributions to the integral together, and get
\be \label{eq:multsineres1}
	\exp \left [ \int_{\mathbb{R}+i0}  \frac{e^{zt}}{t} \sum_{n\in C\cap \mathbb{Z}^r } e^{t(n\cdot\underline \omega ) } dt \right ] = \prod_{f \in \Delta_1^C } 
	( e^{ 2\pi i  \frac{z}{ (\tilde K_\rho\go )_1 } } | e^{2\pi i  \frac{ (\tilde K_\rho\go )_2 }{ (\tilde K_\rho\go )_1 } } , \ldots,e^{2\pi i  \frac{ (\tilde K_\rho\go )_{r} }{ (\tilde K_\rho\go )_1} } )_\infty 
\ee
which of course is exactly 
\be
	\prod_{\rho \in \Delta_1^C } ( x_\rho | \underline q_\rho )_\infty,
\ee
using the definitions in the statement of the theorem. 
This gets us to the first statement in the theorem, under the restrictions on $z$, but we can conclude it for any $z$ away from the poles by analytic continuation. 
The second part is proven by the same methods applied to the other integral representation, restricting $z$ so that the second part of the lemma applies and closing the contour around the lower half plane instead. 
\end{proof}

From this result, we can prove an interesting modular property of the ordinary multiple elliptic gamma functions, which includes and generalizes the usual modular properties for the ordinary $G_r$ functions. 
\begin{proposition}
Let $C$ be a good cone of dimension $r \geq 2$, and $\im \left ( \frac{ (\tilde K_\rho \underline \go )_i } { (\tilde K_\rho \underline \go)_r } \right ) \neq 0$, for all $\rho$ and $i = 1, \ldots, r-1$. 
Then the multiple elliptic gamma function satisfies the identity 
\be
	\prod_{f \in \Delta_1^C} G_{r-2} ( \frac{z}{ (\tilde K_\rho \underline \go)_1 } | \frac{ (\tilde K_\rho \underline \go )_2 } { (\tilde K_\rho \underline \go)_1 }, \ldots, \frac{ (\tilde K_\rho \underline \go )_{r} } { (\tilde K_\rho \underline \go)_1 } )
	= \exp \left [ - \frac{2\pi i }{r!} B_{r,r}^C ( z | \underline \go ) \right ] . 
\ee
\end{proposition}
\begin{proof}
We simply compare the two different factorization results of theorem \ref{thm:SrCproductrep}, raising them to the power $(-1)^r$ and then dividing them with each other we arrive at 
\[	\begin{split}
	\exp \left [ - \frac{ 2 \pi i }{r!} B_{r,r}^C ( z | \underline \go ) \right ] 
	= \prod_{\rho \in \Delta^C_1} \left [ ( x_\rho | \underline q_\rho )_\infty  ( x_\rho^{-1} | \underline q_\rho^{-1} )^{-1}_\infty \right ]^{(-1)^r } \\
	= \prod_{\rho \in \Delta^C_1} G_{r-2} \left ( \frac{z}{ (\tilde K_\rho \underline \go )_1 } | \frac{ (\tilde K_\rho \underline \go)_2 }{(\tilde K_\rho \underline \go)_1 }, \ldots, \frac{ (\tilde K_\rho \underline \go )_{r} } {(\tilde K_\rho \underline \go )_1 } \right ) . 
	\end{split}
\]	
where we've used the definition \eqref{eq:G2def2} of $G_r$. 
\end{proof}
One easily sees that when $C = \mathbb{R}^r_{\geq 0}$ this becomes exactly the ordinary property of theorem \ref{thm:GrModularity}, whereas for other cones it gives us more general relations between the ordinary multiple elliptic gamma functions.

\section{Generalized multiple elliptic gamma functions}

Let $C$ be a good cone of dimension $r$, and let $\im z > 0 $, and $\underline \tau \in \comp^r$ such that $\im \left ( \frac{ (\tilde K_\rho \underline \tau)_j } { (\tilde K_\rho \underline \tau)_1 } \right ) \neq 0$ for all generators $\rho$ of $C$. 
Then define the generalized multiple $q$-polylogarithm associated to $C$ as 
\be
	\li^C ( z | \underline \tau ) = \sum_{n=1}^\infty \frac{e^{2\pi i z n  } } { n } \sum_{m \in C \cap \mathbb{Z}^r } e^{2\pi i n (m\cdot \underline \tau ) } ,
\ee
mimicking one of the definitions of the `usual' multiple $q$-polylogarithm, so that if $C$ is the standard cone our version reduces to the ordinary one. 
We also let $\li^{C^\circ}$ be given by the same definition but instead summing over $C^\circ$.
These two functions are then used to define the generalized $q$-factorials associated to $C$:
\bea
	( x | \underline q )^C_\infty &= \exp ( - \li^C ( z | \underline \tau ) ), \\
	( x | \underline q )^{C^\circ}_\infty &= \exp ( - \li^{C^\circ} ( z | \underline \tau ) ).
\eea
With this, we can again mimic the definition of the usual multiple elliptic gamma function and define our generalized version associated to $C$:
\begin{definition}
For a good cone $C$ of dimension $r$ we define the associated generalized elliptic gamma function as 
\be
	G_{r-1}^C ( z | \underline \tau ) = \left \{ ( x | \underline q )^C_{\infty} \right \}^{(-1)^{r-1} }  ( x^{-1} | \underline q )_\infty^{C^\circ}. 
\ee
\end{definition}
Again, if $C$ is the standard cone $\mathbb{R}^r_{\geq 0 }$, then these generalized gamma functions are identical to the usual multiple elliptic gamma functions $G_r$. 
We also note that if $\im ( \underline \go ) \in \check C^\circ$ one can see from the definition above that $G_{r-1}^C$ has the following infinite product representation
\be
	G_{r-1}^C (z | \underline \tau ) = \prod_{n \in C \cap \mathbb{Z}^r } ( 1 - e^{2\pi i ( z + n\cdot \underline \tau ) } )^{(-1)^{r-1} } \prod_{n \in C^\circ \cap \mathbb{Z}^r } ( 1 - e^{2\pi i ( - z + n\cdot \underline \tau ) } ) .
\ee
The $G_r^C$ functions satisfies a few functional relations similar to the ones for the usual multiple elliptic gamma functions, and it also has a similar integral representation, as well as a kind of modular properties inherited from the cone, which we will show in the following sections.
But we note that many of the relations satisfied by the ordinary elliptic gamma functions, such as the `periodicity' relation
\[
	G_r ( z + \tau_j | \underline \tau )   =  G_{r-1} ( z | \underline \tau^{-} ( j )  ) G_{r} ( z | \underline \tau ), 
\]
do not hold for these generalized ones, since they can be understood as coming from symmetries of the standard cones $\mathbb{R}^n_{\geq 0}$, so for less symmetric cones, the corresponding relation is much more complicated.
Two properties that can easily be seen to still hold are
\[
	\begin{split}
		G_r^C ( z | \underline \tau ) &= \frac { 1 } { G_{r}^C ( - z | -\underline \tau ) } , \\
		G_r^C ( z + 1 | \underline \tau ) &= G_r^C(z | \underline \tau ). 
	\end{split}
\]
Next, we give an integral representation, which we will then use to prove a factorization result, or a modular property, of $G_r^C$.

\subsection{Integral representation}
We start by noting that $\li^C$ has the following integral representation.
\begin{proposition}
For $\im z > 0$ and $\im \tau_j \neq 0 \ \forall j$,
\be
	\li^C ( z | \underline \tau ) = - \int_{C_1} \frac{ e^{2\pi i z t }} { t ( 1 - e^{2\pi i t } ) } \sum_{m \in C \cap \mathbb{Z}^r } e^{2\pi i t (m\cdot \underline \tau) } dt \ ,
\ee
where $C_1$ is the contour shown in figure \ref{fig:C1cont}. 
The same expression also holds for $\li^{C^\circ}$, just changing to summing over $C^\circ$ instead. 
\end{proposition}
\begin{proof}
This is directly verified by summing over all the residues from the poles inside the contour.
\end{proof}
This directly gives us an integral expression for $G_{r-1}^C$: 
\[
	G_{r-1}^C (z | \underline \tau) = \exp \left \{ \int_{C_1}  \left [ (-1)^r \frac{e^{2\pi i z t }}{t ( 1 - e^{2\pi i t )} }\sum_{m\in C\cap\mathbb{Z}^r } e^{2\pi i t (m\cdot \underline \tau) }   -
	  \frac{e^{-2\pi i z t }}{t ( 1 - e^{2\pi i t }) }\sum_{m\in C^\circ \cap\mathbb{Z}^r } e^{2\pi i t (m\cdot \underline \tau) }  \right ] dt \right \} ,
\]
which we will use in the proof of the modularity property of $G_r^C$. 

\subsection{Factorization or modularity property}
We now prove an interesting factorization property of $G_r^C$, that is very similar to the factorization property of $S_r^C$.
One can also view it as something like a modular property, generalizing the modular property of the normal $G_r$-functions, with a dependence on the cone $C$. 

\begin{theorem} \label{thm:GrCfactorization}
Let $C$ be a good cone of dimension $r$, and let $\underline \tau \in \comp^r$ be such that $Im \left( \frac{ (\tilde K_\rho \underline \tau )_j } { (\tilde K_\rho \underline \tau )_1 } \right ) \neq 0$ for all generators $\rho$ of $C$ and for $j=2,\ldots,r$. 
Then 
\bea
	G_{r-1}^C ( z | \underline \tau ) &=& \exp \left [ \frac{2\pi i }{ (r+1)!} B_{r+1,r+1}^{\hat C} ( z | \underline \tau,-1)\right ] \prod_{\rho \in\Delta^C_1} (SK_\rho)^* G_{r-1} (z|\underline \tau) \\
	&=& \exp \left [ - \frac{2\pi i }{ (r+1)!} B_{r+1,r+1}^{\hat C} ( z | \underline \tau,1)\right ] \prod_{\rho\in\Delta^C_1} (S^{-1}K_\rho)^* G_{r-1} (z|\underline \tau),
\eea
where $K_\rho$ is the $SL_{r+1}(\mathbb{Z})$ defined by the generator $\rho$, as described in section \ref{sec:conemodularity}, $S$ is the element defined in \eqref{eq:Sdefinition}, and where by the notation $(SK_\rho)^*$ we mean that the group element acts on the parameters of the function, $(z|\underline \tau)$, as a fractional linear transformation, described by equation \eqref{eq:groupaction1}. 
$B_{r+1,r+1}^{\hat C}$ is the generalized Bernoulli polynomial associated to the $(r+1)$-dimensional cone $\hat C = C \times \mathbb{R}_{\geq 0}$.

\end{theorem}
\begin{proof}
%
%
%
%
%
%
In the following, we use the fact that
\[
	\li^C( z + 1 | \underline \tau ) = \li^C ( z | \underline \tau), 
\]
something that is obvious from its definition. 
This lets us shift $z\rightarrow z-1$ in the second term of the integral representation of $G_{r-1}^C$:
\be 
	G_{r-1}^C (z | \underline \tau) = \exp \left \{ \int_{C_1}  \left [ (-1)^r \frac{e^{2\pi i z t }}{t ( 1 - e^{2\pi i t )} }\sum_{m\in C\cap\mathbb{Z}^r } e^{2\pi i t (m\cdot \underline \tau) }   -
	  \frac{e^{2\pi i(1 - z)t }}{t ( 1 - e^{2\pi i t }) }\sum_{m\in C^\circ \cap\mathbb{Z}^r } e^{2\pi i t (m\cdot \underline \tau) }  \right ] dt \right \} .
\ee
The point of this shift is that it makes it manifest, together with lemma \ref{lem:conesum1}, that the integrand here is an odd function of $t$. 
This fact lets us change the path of integration into either $\mathbb{R}+i\epsilon$ or $-\mathbb{R} - i \epsilon$, and then subsequently close it around the upper or lower half-plane, letting us evaluate it by collecting the poles in either half-plane.
Of course we still need to take care of the additional contribution from the pole at the origin: this will again give us the generalized Bernoulli polynomial just as for the factorization of $S_r^C$. 
Indeed, the rest of the proof follows very much the same lines as that proof: one chooses $z$ so that we can close the contour around say the upper half plane, and then one collect the residues, arguing again that the only proper poles are the ones corresponding to generators of $C$ by using lemma \ref{lem:rayexclusion}, and that their values do not depend on the choice of subdivision.
Then one computes and adds up their residues in the exact same way. 
The first term of the integrand is in fact the same as for the $S_r^C$ computation, and gives for each generator $\rho$ a factor of 
\[
	\left ( e^{2\pi i \frac{z}{ (\tilde K_\rho \underline \tau)_1 } }| e^{2\pi i \frac{1}{ (\tilde K_\rho \underline \tau)_1 } },  e^{2\pi i  \frac{(\tilde K_\rho \underline \tau)_2}{ (\tilde K_\rho \underline \tau)_1 } } , \ldots ,   e^{2\pi i  \frac{(\tilde K_\rho \underline \tau)_2}{ (\tilde K_\rho \underline \tau)_1 }} \right )_\infty^{(-1)^{r}} ,
\]
The second term similarly gives 
\[
	\left ( e^{-2\pi i \frac{z}{ (\tilde K_\rho \underline \tau)_1 } }| e^{-2\pi i \frac{1}{ (\tilde K_\rho \underline \tau)_1 } },  e^{-2\pi i  \frac{(\tilde K_\rho \underline \tau)_2}{ (\tilde K_\rho \underline \tau)_1 } } , \ldots ,   e^{-2\pi i  \frac{(\tilde K_\rho \underline \tau)_2}{ (\tilde K_\rho \underline \tau)_1 }} \right )_\infty^{(-1)^{r-1}} , 
\]
and combining them gives us an ordinary $G_{r-1}$ function: 
\be \label{eq:modularGr}
	G_{r-1} \left ( -\frac{z}{ (\tilde K_\rho \underline \tau)_1 } | - \frac{1}{ (\tilde K_\rho \underline \tau)_1 }, -  \frac{(\tilde K_\rho \underline \tau)_2}{ (\tilde K_\rho \underline \tau)_1 }, \ldots, -\frac{(\tilde K_\rho \underline \tau)_2}{ (\tilde K_\rho \underline \tau)_1 } \right ) = (S^{-1} K_\rho )^* G_{r-1} ( z | \underline \tau ) , 
\ee
where $(S^{-1} K_\rho)\in SL_{r}(\mathbb{Z})$ acts on the parameters of $G_{r-1}$ as a modular transformation defined in equation \eqref{eq:groupaction1}.

Finally, the contribution from the origin is given by
\be
	\frac{1}{2} \int_{C_0}  \left [ (-1)^r \frac{e^{2\pi i z t }}{t ( 1 - e^{2\pi i t )} }\sum_{m\in C\cap\mathbb{Z}^r } e^{2\pi i t (m\cdot \underline \tau) }   -
	  \frac{e^{2\pi i(1 - z)t }}{t ( 1 - e^{2\pi i t }) }\sum_{m\in C^\circ \cap\mathbb{Z}^r } e^{2\pi i t (m\cdot \underline \tau) }  \right ] dt , 
\ee
where $C_0$ again is a small circle around the origin. 
Now as before it's easy to see that the two terms both will give us exactly the generalized Bernoulli polynomials of $\hat C$ and $\hat C^\circ$ respectively, where $\hat C = C \times \mathbb{R}_{\geq 0}$. 
The `extra dimension' comes from the factor of $1/(1-e^{2\pi i t } )$ in front, which also gives us the parameters $(\underline \tau, 1 )$. 
And finally by applying the property \eqref{eq:bernoulliCproperty} of our generalized Bernoulli polynomials, we see that the two terms actually give the same result so they combine nicely, and we find that the integral above is equal to  
\[
	-\frac{2\pi i}{(r+1)!} B_{r+1,r+1}^{\hat C} ( z | (\underline \tau, 1) ). 
\]	
Combining this with the result of equation \eqref{eq:modularGr} gives us the second statement of the theorem, and the first version is proved in a completely analogous way, except choosing to modify to modify the contour of integration differently and close the integral around the lower half-plane instead. 
And analytic continuation again gives us that the statement holds for any $z$ away from the poles.

%

\end{proof}
%
%

Finally we give the following theorem which comes from the integral representation of $G^C_{r}(z|\underline \tau)$ and $S_r^C(z|\underline \go)$; or that can also be proven using the factorization results of both $G_r^C$ and $S_r^C$: comparing the two factorization results also makes it clear why the Bernoulli polynomials appear in this expansion.
We also use that both $B_{r,r}^C$ and $S_r^C$ are invariant under $(z|\underline\go) \mapsto (cz| c\underline \go )$ for the proof. 
\begin{theorem}[Representation of $G_r^C$ by an infinite product of $S_{r+1}^C$]
Assume that $\im \tau_j > 0 \ \forall j$, then
\be
	\begin{split}
		G_r^C ( z | \underline \tau ) =& \exp \left [ \frac{2\pi i }{(r+2)! } B_{r+2,r+2}^{\hat C} ( z | \underline \tau, -1 ) \right ] \\
		&\times \prod_{k=0}^\infty \frac{ S_{r+1}^C ( z + k + 1 | \underline \tau )^{(-1)^r } S_{r+1}^C ( z - k  | \underline \tau )^{(-1)^r } } { \exp \left [ \frac{\pi i }{(r+1)! } (B_{r+1,r+1}^C ( z+k+1|\underline \tau ) - B_{r+1,r+1}^C ( z-k|\underline \tau )) \right ] }
	\end{split}
\ee
\end{theorem}

\providecommand{\href}[2]{#2}\begingroup\raggedright\endgroup

\end{document}